\documentclass[12pt]{article}%{amsart}%

\usepackage[ansinew]{inputenc}

\usepackage{latexsym,amsmath,amsthm,amssymb,amsfonts,amscd}
\usepackage{graphicx}
\usepackage{color}

\newtheorem{theorem}{Theorem} 
\newtheorem{lemma}{Lemma}

\theoremstyle{definition}
\newtheorem{rmk}{Remark}
\newtheorem{conj}{Conjecture}

\newcommand{\comment} [1]{}
\def\ep{\varepsilon}

\def\la{\lambda}

\def\cA{\mathcal A}
\def\cB{\mathcal B}
\def\cC{\mathcal C}

\def\cE{\mathcal E}

\def\cF{\mathcal F}

\def\cK{\mathcal K}

\def\cW{\mathcal W}

\def\cP{\mathcal P}

\newcommand{\indic}[1]{\mathbf{1}_{\{#1\}}}

\newcommand{\ds}{\displaystyle}

\def\mn{\medskip\noindent}

\def\beq{\begin{equation}}
\def\eeq{\end{equation}}
\def\beqa{\begin{eqnarray}}
\def\eeqa{\end{eqnarray}}
\def\beqax{\begin{eqnarray*}}
\def\eeqax{\end{eqnarray*}}
\def\sqz{\kern -0.2em}

\def\eps{\varepsilon}

\def\P{\mathbb{P}}
\def\R{\mathbb{R}}
\DeclareMathOperator{\bes}{({Bess})}
\def\bess{P^{\bes}}
\newcommand{\bm}{\mathbb{W}}

\begin{document}

\title{An integral test for the transience of a Brownian path with limited local time}

\author{Itai Benjamini$^1$ and Nathana\"el Berestycki$^2$}

\date{April 21, 2010}
\maketitle

\centerline{\textbf{Abstract.}}

\mn

\mn
We study a one-dimensional Brownian motion conditioned on a self-repelling behaviour. Given a nondecreasing positive function
$f(t), t\ge 0$, consider the measures $\mu_t$ obtained by conditioning a Brownian path so that $L_s\le f(s)$, for all $s\le t$, where $L_s$ is the local time spent at the origin by time $s$.
It is shown that the measures $\mu_t$ are tight, and that any weak limit of $\mu_t$ as $t\to \infty$ is transient provided that $ t^{-3/2}f(t)$ is integrable. We conjecture that this condition is sharp and present a number of open problems.

\mn
\centerline{\textbf{R\'esum\'e.}}

\mn

\mn
Etant donn\'ee une fonction cro\^issante
$f(t), t\ge 0$, consid\'erons la mesure $\mu_t$ obtenue lorsqu'on on conditionne un mouvement brownien de sorte que $L_s\le f(s)$, pour tout $s\le t$, o\`u $L_s$ est le temps local accumul\'e au temps $s$ \`a l'origine.
Nous montrons que les mesures $\mu_t$ sont tendues, et que toute limite faible de $\mu_t$ lorsque $t\to \infty$ est la loi d'un processus transient si $t^{-3/2}f(t)$ est int\'egrable. Nous conjecturons que cette condition est \'egalement n\'ecessaire pour la transience et proposons un certain nombre de questions ouvertes.

\vfill

\mn 1.
 \texttt{itai.benjamini@weizmann.ac.il}. Weizmann Institute of
 Science. Rehovot, Israel.

\mn 2.
\texttt{N.Berestycki@statslab.cam.ac.uk}. University of Cambridge. CMS, 3 Wilberforce Road, Cambridge, CB3 0WB, United Kingdom.

\newpage

\section{Introduction}

Let $(X_t,t\ge 0)$ be a Brownian motion in $\R^d$. It is well-known that $d=2$ is a critical value for the recurrence or transience of $X$. In this paper, we show however that even in dimension 1, a very small perturbation of the Brownian path may result in the transience of the process.  Let $f:[0,\infty) \to [0,\infty)$ be a given nonnegative Borel function such that $f(0) >0$, and consider the event
$$
\cK_t = \left\{ L_s \le f(s), \ \ \text{ for all } s\le t\right\}.
$$
Here $(L_s,s\ge 0)$ is a continuous determination of the local time process of $X$ at the origin $x=0$. Our goal in this paper is to analyse the limiting behaviour of the Wiener measure conditioned on $\cK_t$, as $t\to \infty$. Since $(L_s,s\ge 0)$ is almost surely nondecreasing, we may and will assume without loss of generality that $f$ is nondecreasing. (Otherwise one may always consider $\hat f(t) = \inf_{s\ge t} f(s)$). Note that by Brownian scaling, $L_t$ is of order $\sqrt{t}$ for an unconditional Brownian path. Hence when $f(t) \le t^{1/2}$, this constraint is of a self-repelling nature, since it forces the Brownian path to spend less time than it would naturally want to at the origin. We will thus assume that $f$ is nondecreasing and $t^{-1/2}f(t)$ is non-increasing.

Our main result is that if $f$ is only logarithmically smaller than $t^{1/2}$, then $X$ becomes transient almost surely in the limit $t\to\infty$. Here, a probability measure $\P$ on the space $\cC$ of continuous sample paths is called transient (almost surely) if $\P(\lim_{t\to\infty} |X_t| = +\infty) = 1$. If $\limsup_{t\to \infty} X_t = +\infty$ and $\liminf_{t\to \infty} X_t = -\infty$ with $\P$-probability 1, then $\P$ is called recurrent. We equip $\cC$ with the topology of uniform convergence on compact sets, which turns $\cC$ into a Polish space, and discuss weak convergence of probability measures on $\cC$ with respect to this topology.

\begin{theorem}\label{T:pt} Let $\bm$ denote the Wiener measure on $\cC$, and let $\bm_t = \bm(\cdot |\cK_t)$. Then $\{\bm_t,t\ge 0\}$ is a tight family. Assume further that
\begin{equation}
 \label{cond:f}
 \int_1^\infty \frac{f(t)}{t^{3/2}} dt <\infty.
\end{equation}
Then for any weak subsequential limit $\P$ of $\bm_t$ as $t\to\infty$, $\P$ is transient almost surely.
\end{theorem}

In particular, if $f(t) \sim \sqrt{t} (\log t)^{-\gamma}$ with $\gamma\ge 0$, then $\P$ is transient as soon as $\gamma>1$. We believe, but have not succeeded in proving, that condition (\ref{cond:f}) is sharp, in the following sense.

\begin{conj} If
\begin{equation}\label{cond:f2}
\int_1^\infty \frac{f(t)}{t^{3/2}} dt =\infty,
\end{equation}
then any weak limit $\P$ of $\bm_t$ as $t\to\infty$ is recurrent almost surely.
\end{conj}
It should be noted that this problem is open even in the basic case where $f(t) \sim \sqrt{t}$ as $t\to \infty$, for whch it is still the case that $\bm(\cE_t) \to 0$ as $t\to\infty$.

Figure \ref{fig} below illustrates this result.

\begin{figure}\label{fig}
\begin{center}
\includegraphics[scale=.45]{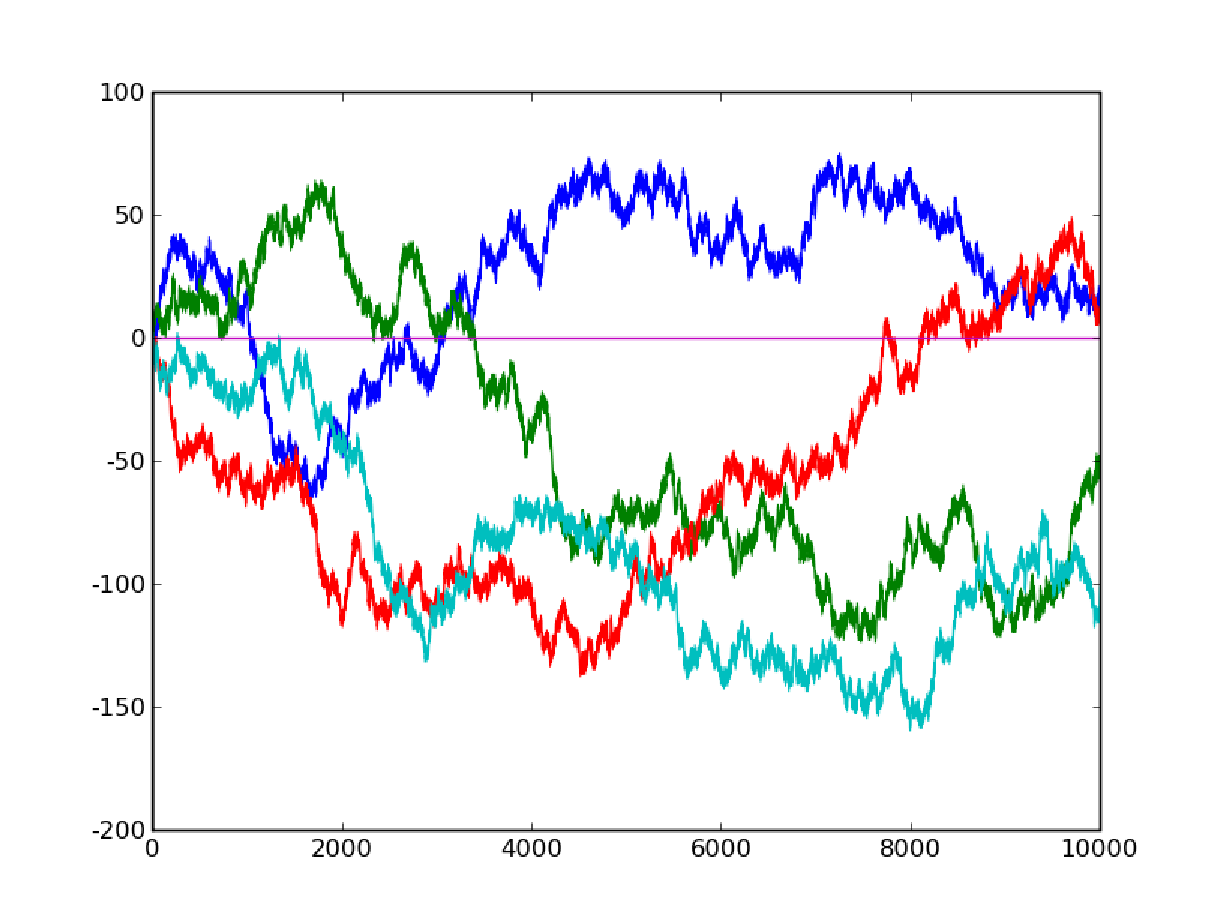}\\ \includegraphics[scale=.45]{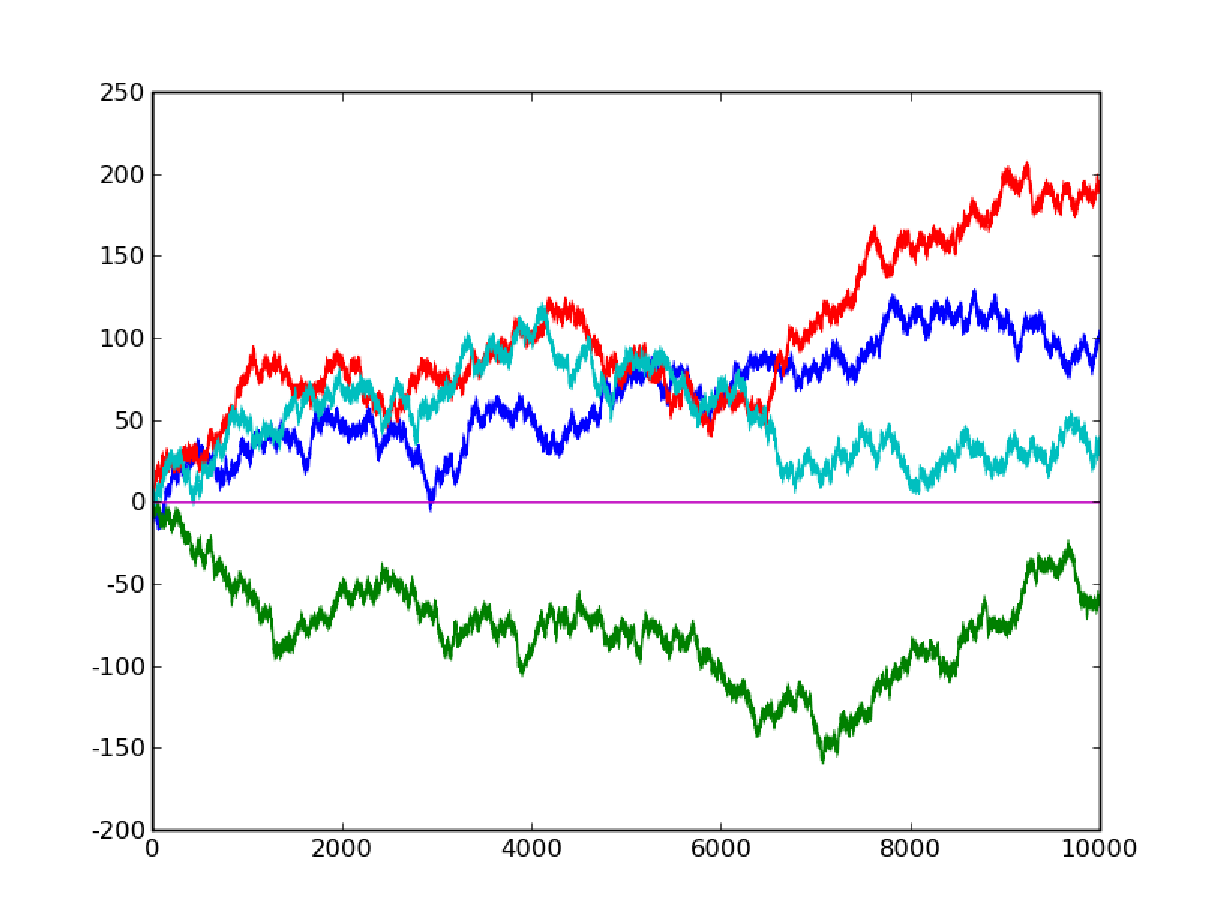}\\ \includegraphics[scale=.45]{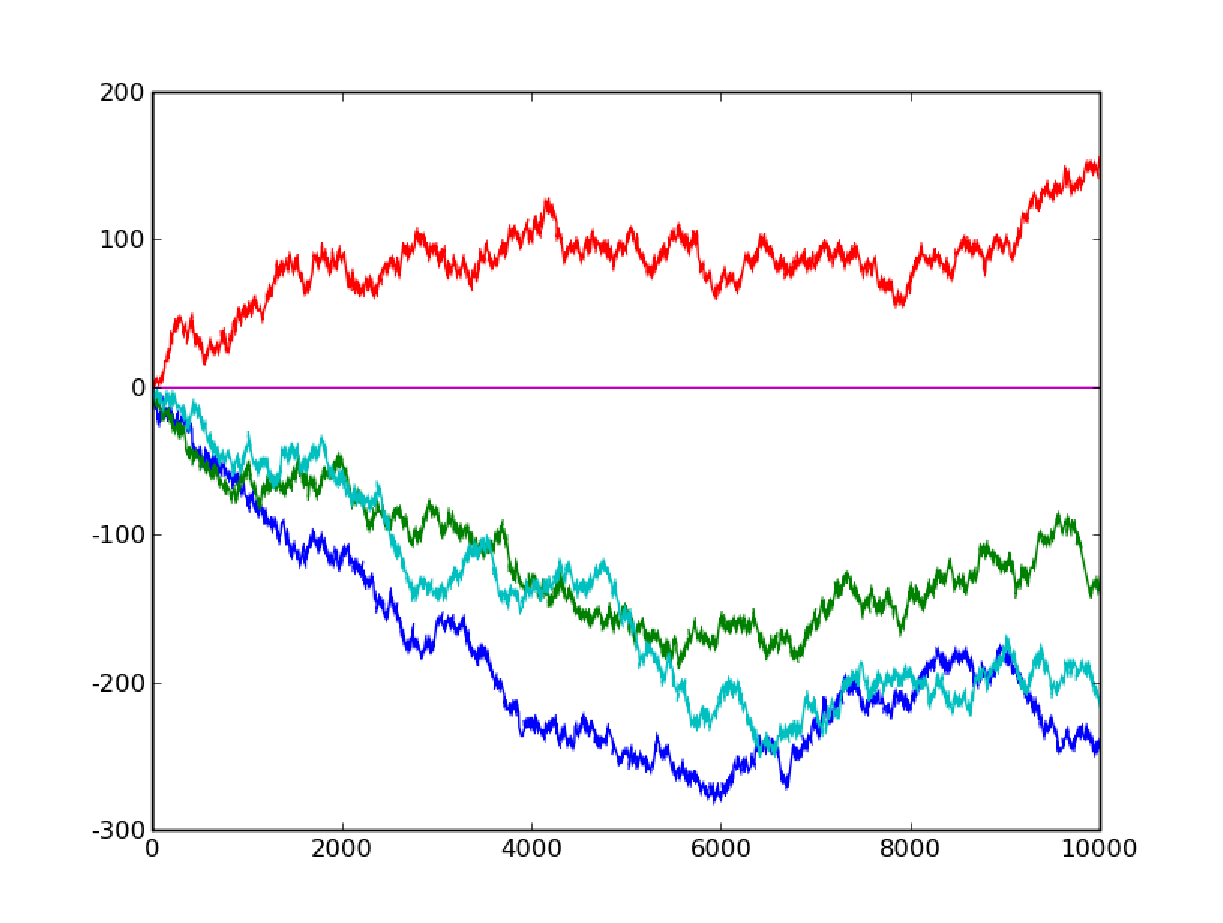}
\caption{Simulations of trajectories up to time $t=10^4$. From
top to bottom: $\gamma=0.5$, $\gamma=0.9$, and $\gamma=1.1$.}
\end{center}
\end{figure}

In the recurrent regime, i.e. if (\ref{cond:f2}) holds, we further believe that the local time process of $X$ is well-defined almost surely under $\P$, but that $X$ is ``far away from breaking the constraint $\cK_t$", in the following sense:

\begin{conj}
 Assume (\ref{cond:f2}), and let $\P$ be any weak limit of $\bm_t$. Then there exists a nonnegative deterministic function $\omega(t) \to \infty$ as $t\to \infty$ such that
 \begin{equation}
 \P\left(L_t \le \frac{f(t) }{\omega(t)}\right) \to 1
 \end{equation}
 as $t\to \infty$.
 \end{conj}

Furthermore, when $f(t)\sim t^{1/2}(\log t)^{-\gamma}$ as $t\to \infty$, with $0<\gamma <1$, we believe that
\begin{equation}
L_t \le f(t) \exp( - C (\log t)^{\gamma}),
\end{equation}
with probability asymptotically 1, for some $C>0$. It may seem surprising at first that, in the recurrent regime, the process shouldn't use its full allowance of local time. This phenomenon is related to entropic effects, which cause the process to stay far away from breaking the constraint to allow for more fluctuations. In \cite{bb}, we already observed a similar behaviour in the case where the local time profile of the process is conditioned to remain bounded at every point, and have called this phenomenon ``Brownian entropic repulsion". This aspect is actually crucially exploited in our proof, which relies on considering a suitable softer constraint $\cK'_t$ (easier to analyse, because more ``Markovian"), but which nevertheless turns out to be equivalent to that of $\cK_t$.

\paragraph{Discussion and relation to previous works.}

\medskip The integral test (\ref{cond:f}) is reminiscent of classical integral tests on Brownian motion. Indeed, an indication that this test provides the right answer follows from a pretty basic calculation. This calculation is carried out in Lemma \ref{L:Ajub}, where the probability of hitting 0 during the interval $[t/2,t]$ given $\cK_t$ is estimated. However we stress that one of the major difficulties of this problem is to control the long-range interactions induced by conditioning far away into the future, and to show that this propagates down to an arbitrarily large but finite window close to the origin in a manageable way. It is this long-range interaction, inherent to the study of self-interacting processes, which is at the source of our difficulties in proving Conjecture 1 and 2.

\medskip Putting a bound on the local time can be viewed as introducing
some form of self-repellence of the process. The problems studied here
are reminiscent of some problems arising in
the mathematical study of
\emph{random polymers}, of which an excellent review can be found in \cite{polreview}. Our work is also somewhat related in spirit to
a series of papers by Roynette et al. (see, e.g., \cite{brv},
or the forthcoming monograph by Roynette and Yor \cite{roynette-yor}) and Najnudel \cite{najnudel} , although our goals and methods are quite different.

\medskip Theorem \ref{T:pt} establishes tightness of the measures $\mathbb{P}_t$ but not weak convergence. One possible approach to prove uniqueness would be to identify the limiting process as the unique solution to a certain stochastic differential equation. We note that this is related to the work of Barlow and Perkins \cite{bp}, which describes the behaviour of Brownian motion near a typical slow point, i.e., a time $t$ near which the growth of $B$ satisfies
$$
\limsup_{h\to 0} h^{-1/2}|B_{t+h}-B_t|  =1.
$$
Blowing up the trajectory near this slow point, their Theorem 3.3 gives precisely a description of the process as a solution to a certain stochastic differential equation.

\paragraph{Organization of the paper.} In section \ref{S:prelim}, we prove some preliminary results which contain results interesting for their own sake. Namely, it is shown in Theorem \ref{T:bdd} that a Brownian motion conditioned on having a local time at the origin bounded by 1 is transient, and that the total local time accumulated by this process is a uniform random variable on $(0,1)$. Note that this is smaller than 1 almost surely, so here again the process doesn't use its full allowance of local time. Also, in Theorem \ref{T:slow} that a Brownian motion conditioned on the event $\cE_t = \{L_t \le f(t)\}$ is recurrent in the limit $t\to \infty$ as soon as $f(t) \to \infty$, no matter how slowly.

In section \ref{S:proof}, we give a proof of the main result (Theorem \ref{T:pt}). This is based partly on Theorem \ref{T:slow} and on a general result which shows that any conditioning of the Brownian motion based on its zero set cannot grow faster than diffusively (Lemma \ref{L:subdiff}) and, in the case of $\cK_t$, this is matched by a lower bound of the same order of magnitude (Lemma \ref{L:supd}). These various ingredients are put together using a coupling method, which then gives the proof of the result.

Finally in section \ref{S:neg}, we study a slightly different but related problem, where a Brownian path is conditioned to spend no more than one unit of time in the negative half-line. It is shown there again that the measures converge weakly to a limiting process, which is (unsurprisingly) transient, and also that the total amount of time spent in the forbidden region by this process is equal to $U^2$, where $U$ is a uniform random variable on (0,1). Hence here again, the process does not use its full allowance, another expression of the entropic repulsion principle.

\section{Preliminaries}
\label{S:prelim}

\subsection
{Brownian motion with bounded local time.}

It will be convenient to define various processes on the
same space, but governed by different probability measures on this
space. We take for this common space the space $\cC = C([0,
 \infty), \mathbb R)=$ the space of continuous functions from $[0,
\infty)$ into $\mathbb R$. $X_s$ will denote the $s$-th coordinate
function on $\cC$; we shall also write $X(s)$ for $X_s$
occasionally when $s$ is a complicated expression. In this setup
Brownian motion is obtained by putting the Wiener measure $\mathbb
W$ on $\cC$; $\mathbb{W}$ is concentrated on the paths which start at
$X(0)=0$ and makes increments over disjoint intervals independent
with suitable Gaussian distributions. We now take $L(\cdot, \cdot)$ as a jointly continuous local time
of the Brownian motion. This is a continuous function $L(s,x)$
which satisfies
\begin{equation}
\label{loctime}
\big\{|\{s \le t:
X_s \in B\}|\big \} = \int_0^t I[X_s \in  B]ds = \int_{x \in B}
L(t,x)dx 
\end{equation}
$\bm$-almost surely simultaneously for all Borel sets $B$ and $t \ge 0$
(\cite{IkdeaWatanabe} , Sect 3.4). Under the measure $\mathbb W$
there a.s. exists such a jointly continuous function, and it is
clearly unique for any sample function for which it exists.
\comment{ We shall also need the cadlag inverse function of
$L(s,0)$:
$$
\tau_s := \inf\{u:L(u,0) > s\}.
$$
}

\medskip As a first step towards the proof of Theorem \ref{T:pt}, we prove the following simple result. Assume that $f(t) =1$, so that
$$
\cK_t := \{L_t \le 1\}.
$$
Our first theorem describes the weak
limit of $\bm_t$ as $t \to \infty$. This description
involves a Bessel-3 process, a description of which can be found,
for instance, in \cite{revuz-yor}.

\begin{theorem}\label{T:bdd}
The measures $\bm_t$ converge weakly on $\cC$ to a measure
$\mathbb{P}$. Under $\mathbb{P}$, the process $X$ is transient and
$\mathbb P$ can be described as follows: On some probability space
let $U, \{B(s), s\ge 0\}, \epsilon$ and $\{B^{(3)}(s), s \ge 0\}$
respectively be a random variable with a uniform  distribution on
$[0,1]$, a Brownian motion, a random variable uniform on $\{-1,1\}$, and a Bessel-3 process, and assume that
these four random elements are independent of each other. Define
\beq
\label{tau(T)} \tau= \sup\{v:L(v,0) < U\}
\eeq
(where $L$ is the local time of $B$), and
$$
Y(t) = \begin{cases} B(t) &\text{ if }t \le \tau\\
\epsilon B^{(3)}(t-\tau) & \text{ if } t > \tau.
\end{cases}
$$
Then $\mathbb P$ is the distribution of $\{Y(t), t \ge 0\}$.
\end{theorem}

Somewhat informally, the theorem says that under $\mathbb P$,
$X$ can be described by first drawing an independent uniform random
variable $U$. Then $X$ is the standard Brownian
motion until it has accumulated a local time at 0 equal to $U$,
and performs a three-dimensional Bessel process afterwards.

It is well known that a Bessel-3 process starting at
  the origin diverges to infinity almost surely. This is of course the reason
  why the process governed by $\mathbb P$ is transient. However, we can
  say more. It is also well known that
  $L_t$, the
  local time at 0 can change only at times $t$ when $X_t=0$. This fact
  is also clear from \eqref{loctime}. Together with the description of
  the process under $\mathbb P$ this implies that $L_t$ is
  a.s. constant on $t \ge \tau$ at which it takes the value $U$
(by definition and continuity of the inverse local time $\tau$).
Thus, the theorem implies
\beq \label{rigid}
L_\infty = L_\tau = U.
\eeq
Since $U < 1$ almost surely, this shows that under $\mathbb P$,
$X$ does \emph{not} use its full allowance of local time, which is another expression of the entropic repulsion principle.

\emph{Proof of Theorem \ref{T:bdd}}.

\mn {\bf Step 1}. In this step we shall give a representation of
Brownian motion by means of excursions. This will turn out to be
useful for the proof. Readers familiar with this sort of things
are encouraged to skip this step and go to step 2. To help with
the intuition, consider the set $Z:=\{t:X_t= 0\}$. If $X_t$ is a
continuous function of $t$, then $Z$ is a
  closed set, and its complement, $\mathbb R\setminus Z$ is a
  countable union of maximal open intervals. On each such interval $X
  \ne 0$. The piece of the path of $X$ on such an interval is called an
  {\it excursion} of $X$. One can now try to construct a process equivalent
  to $X$ by first picking excursions on some probability space and
according to a suitable distribution, and then putting these
excursions together. For $X$ a Brownian motion, this can be done
rather explicitly. The following description can be found in a
number of references (see, e.g., \cite[Section
III.4.3]{IkdeaWatanabe}, \cite[Chapter XII]{revuz-yor}). The excursions are
elements of $\cW$ which is the collection of continuous functions
$w:[0,\infty) \to
  \mathbb R$ such that $w(0) = 0$, and for which there exists a
  $\zeta(w) > 0$ such that $w(t) > 0$ or $w(t) < 0$ for all $0 < t <
  \zeta(w)$ and $w(t) = 0$ for $t \ge \zeta(w)$. $\zeta(w)$ is called the
  {\it length of the excursion} $w$, or sometimes its duration.

It\^o's fundamental result about the excursions of a Brownian motion
states that there exists a $\sigma$-finite measure $\nu$ on the
space $\mathcal{W}$, called \emph{Itô's measure}, such that the
Brownian motion, viewed in the correct time-scale, can be seen as
a Point process of excursions with intensity measure $\nu$. To
state this result precisely, note that $X$ has only countably many
excursions (since there are only finitely many excursions above
$1/n$ for each finite $n\ge1$ in any compact time-interval). Let
$(e_i)_{i=1}^{\infty}$ be an enumeration of these excursions.
Since $L_t$ only increases on the zero set of $X$, let $\ell_i$ be
the common value of $L_t$ throughout the excursion $e_i$, for
$i\ge 1$. Then Itô's theorem states that:
\begin{equation}\label{PPP}
M:=\sum_{i\ge 1} \delta_{(\ell_i,e_i)}
\end{equation}
is a Poisson point process on $(0,\infty)\times \mathcal{W}$, with
intensity measure $d\la \otimes d\nu$, where $d\lambda$ is the
Lebesgue measure on $(0,\infty)$. That is, for any Borel set in
$(0,\infty)\times \mathcal{W}$, if
$$
\cP(B) := \text{(the number
of points of this process in the set $B$)},
$$
then $\cP(B)$ has a Poisson distribution with mean $\int_B d\la
d\nu$, and for disjoint sets $B_i$ in $(0,\infty) \times
\mathcal{W}$, the $\cP(B_i)$ are independent. With a slight abuse
of notation, we say that $(l,e)\in M$ if $M({l,e})=1$. Note that
the collection of points $(\ell_i,e_i)$ entirely determines the
path of $X$. [Indeed, if we define, for all $u>0$,
$$
\tau(u)=\sum_{\underset{\ell\le u}{(\ell,e)\in M}} \zeta(e),
$$
then for all $i\ge 1$, the function $\tau(u)$ has an upward jump
of size $\zeta(e_i)$ at time $s_i$, and these are the only jumps
of $\tau$. If $t>0$, let $s=\inf\{u \ge t: \Delta \tau(u)>0\}$,
and let $e$ be the excursion associated with the jump of $\tau$ at
time $u$. Then it is easy to check that we have the formula
$$
e(t-\tau(u^-))=X_t
$$
where $X$ is the original process that we started with, and for
all $u>0$,
\begin{equation}
\label{tau} \tau(u)=\inf\{t>0:L_t>u\}.
\end{equation}
Thus the excursions can easily be put together.]

A well-known description of It\^o's measure (see, e.g., \cite{revuz-yor}, XII.4), which we will use in
this proof, is the following determination of the ``law" of the
duration of an It\^o excursion:
\begin{equation}\label{Poisson}
\nu(\zeta>t)=\int_t^\infty \frac{ds}{\sqrt{2 \pi
  s^3}} = \sqrt {\frac 2{\pi t}}.
\end{equation}

\mn \textbf{Step 2.} Consider the event $\mathcal{E}'_t$ that by
time $\tau(1)$ (defined by (\ref{tau})), there is an excursion of
duration greater than $t$. That is, formally:
$$
\mathcal{E}'_t=\{\text{there exists $(\ell,e) \in M$ such that }
\ell \le 1 \text{ and } \zeta(e)>t\}
$$
Observe that on the one hand, $\mathcal{E}'_t \subset
\mathcal{E}_t$. Indeed, $\mathcal{E}_t$ can be written as:
$$
\mathcal{E}_t= \{\tau(1)>t\}
$$
and it is clear that this occurs on $\mathcal{E}'_t$. On the other
hand, we claim that the two events have asymptotically the same
probability. Indeed, if $B=\{e\in \mathcal{W}: \zeta(e)>t\}$, then
$\mathcal{E}'_t$ is the event that by time 1, at least one point
of $M$ has fallen in $B$. Since the number of such points is
Poisson with a parameter given in (\ref{Poisson}), we obtain:
\begin{equation}\label{eq22}
P(\mathcal{E}'_t) =\nonumber 1-\exp\left(-\sqrt {\frac 2{\pi
t}}\right)
\end{equation}
from which we deduce:
\begin{equation}
P(\mathcal{E}'_t) \sim \sqrt {\frac 2{\pi t}} \label{P(E'_t)}
\end{equation}
 as $t\to \infty$. To compute $P(\mathcal{E}_t)$, we appeal to L\'evy's reflection principle. Let
$$
S_t = \sup_{s \le t} X_s.
$$
Then L\'evy showed that under the Wiener measure $\mathbb W$, the
two processes $(L(t,0),t\ge 0)$ and $(S_t,t\ge 0)$ have the same
distribution. On the other hand, for fixed $t\ge 0$, by the
standard reflection principle, $S_t$ has the same distribution as
$|X_t|$ (see, e.g., Durrett \cite{durrett}, Section 7.4, or
\cite{revuz-yor} Sections III.3.7 and VI.2.3). Consequently,
\begin{equation}
\label{D:5} \mathbb W\{\cE_t\}= \mathbb W\{|X_t| \le 1\} = \frac
1{\sqrt {2 \pi
    t}} \int_{x \in [-1,1]}e^{-x^2/2t}dx \sim \sqrt{2/(\pi t)}.
\end{equation}
Let $A\in \mathcal{F}_\infty=\bigvee_{t\ge 0} \mathcal{F}_t$. It
follows from the above that
\begin{equation}
\label{E:simpl1} \mathbb{W}(A|\mathcal{E}_t)\underset{t\to
\infty}\longrightarrow \P(A) \text{ if and only if }
\mathbb{W}(A|\mathcal{E}'_t)\underset{t\to \infty}\longrightarrow
\P(A).
\end{equation}
It thus suffices to prove Theorem \ref{T:bdd} when the event we
condition on is $\mathcal{E}'_t$, rather than $\mathcal{E}_t$. In
fact, for the same reason, one can condition on the event
$\mathcal{E}^{(2)}_t$ that there is exactly one excursion of
duration greater than $t$ prior to time $\tau(1)$. Indeed
$\mathcal{E}^{(2)}_t \subset \mathcal{E}'_t$ and we also have
$P(\mathcal{E}^{(2)}_t)\sim P(\mathcal{E}_t)$ since the
probability that there are two or more such excursion is
$O(t^{-1})$.

Now, for $B\subset \mathcal{W}$, let $(N_u^B,u\ge 0) = \cP([0,u)\times B)$ be the
number of points of $M$ that have fallen in $B$ by time $u$, and take $B$ to be the set $B=\{ e \in \cW: \zeta(e)>t\}$ defined above (\ref{eq22}). Note
that $N^{B}_\cdot $ is a Poisson process with rate $\nu(B)$. It is
well-known that, conditionally on the number of jumps of a Poisson
process by time 1, the jump times have the distribution of the
uniform order statistics. In particular, since
$$\mathcal{E}^{(2)}_t=\{N_1^B=1\},$$
we see that conditional upon $\mathcal{E}^{(2)}_t$, there exists a
uniform random variable $U$ in $(0,1)$ and $e\in B$ such that
$(U,e)\in M$. Moreover $e$ is independent of $U$ and is
distributed according to $\nu(\cdot|B)$. That is,
\begin{equation}\label{e}
P(e\in \cdot) = \frac{\nu(e\in \cdot)}{\nu(\zeta>t)}.
\end{equation}
Consider now the conditional distribution of $\sum_{i\ge 1: e_i
\in B^c} \delta_{(\ell_i, e_i)}$ given $\mathcal{E}^{(2)}_t$. By
the independence property of Poisson point processes in disjoint
sets, this distribution is simply equal to the unconditional
distribution of the restriction of $M$ to $(0,\infty)\times B^c$.
Therefore, let $M'$ be an independent realization of $M$, and let
\begin{equation} \label{tildeM}
\tilde M = M'|_{(0,1]\times B^c}+ \delta_{(U,e)} + M|_{(1, \infty] \times \cW}
\end{equation}
Let $\tilde X$ be the process obtained by reconstructing the path
from the point process $\tilde M$. The above reasoning shows that
for a set $A \in \mathcal{F}_s$, where $s>0$ is fixed (while $t>s$ tends to infinity,)
\begin{eqnarray*}
P(X\in A|\mathcal{E}_t) &\sim & P(X\in A | \mathcal{E}^{(2)}_t)\\
&\sim & P(\tilde X \in A)
\end{eqnarray*}
Thus it suffices to show that $\tilde X$ converges in distribution
to the law $Q$ of the process $Y$ in Theorem \ref{T:bdd}. Note
that $(\tilde X_t,0\le t \le \tau(U))$ depends only on the points
of $M'$, and is thus independent of $(U,e)$. Moreover, provided
$M'$ did not have any point in $B$ on the time-interval $[0,1]$ (an event of probability
$1-o(1)$), $(\tilde X_t,0\le t \le \tau(U)$ has the same
distribution as $(X_t,0\le t \le \tau(U))$. Since $e$ is also
independent from $U$, it thus suffices to prove that
\begin{equation}
\label{exc-bess} \nu(\cdot | \zeta>t) \underset{t\to
\infty}{\longrightarrow} P(\epsilon B^{(3)} \in \cdot)
\end{equation}
weakly. There are many ways to prove (\ref{exc-bess}), and we propose one below. Let us postpone the proof of this statement for a few moments and finish the proof of Theorem \ref{T:bdd}. What we have proved is that for every $s>0$, \begin{equation}\label{conv6}
\bm(A|\cE_t) \to \mathbb{P}(A), \ \ A \in \cF_s, \ \ s>0
\end{equation}
where $\mathbb{P}$ is the measure described in the statement of Theorem \ref{T:bdd}. It is not hard (but not immediate) to deduce weak convergence of $\bm(\cdot | \cE_t)$ towards $\mathbb{P}$. The problem is that one cannot directly apply the $\lambda-\pi$ system theorem of Dynkin to conclude that (\ref{conv6}) holds for all $A \in \cF_\infty$. Instead, note that, as in the proof of Lemma 6 in \cite{bb}, (\ref{conv6}) implies tightness (all events involved in the verification of tightness are measurable with respect to $\cF_s$ for some $s>0$), and furthermore any weak limit must be identical to $\mathbb{P}$, because for instance of the convergence of the finite-dimensional distributions. Thus $\bm_t$ converges weakly towards $\mathbb{P}$.

Turning to the proof of (\ref{exc-bess}), which is well-known in the folklore (but we haven't been able to find a precise reference), we
propose the following simple argument. First note that under $\nu$, $\text{sign($e$)}$ is uniform on $\{-1,+1\}$ and is independent from $(|e(x)|,x \ge 0)$, which has a ``distribution" equal to $\nu^+$, the restriction of $\nu$ to positive excursions. Let $\nu^+(\cdot|\zeta=t)$
denote the law of a positive It\^o excursion conditioned to have duration equal
to $t$, that is, the weak limit of
\begin{equation}
\frac{\nu^+(\cdot ; \zeta\in (t,t+\eps))}{\nu^+(\zeta \in (t,t+\eps))}
\end{equation}
as $\eps\to 0$. Since
\begin{equation}
\nu^+(A|\zeta>t)=\int_{s>t} \nu(A|\zeta=s)\frac{ds}{\sqrt{2\pi
s^3}},
\end{equation}
it suffices to prove
\begin{equation}
\label{exc-bess+} \nu^+(\cdot | \zeta=t) \underset{t\to
\infty}{\longrightarrow} P(B^{(3)} \in \cdot)
\end{equation}
It is not hard to show (see,
e.g., Pitman \cite{pitman-sde}, formula (28)), that a Brownian
excursion conditioned to have duration equal to $t$ is equal in
distribution to a 3-dimensional Bessel bridge of duration $t$,
that is, can be written as
\begin{equation}\label{exc-bes2}
e_u=\sqrt{b_{1,u}^2+b_{2,u}^2+b_{3,u}^2}, \ \ 0\le u\le t,
\end{equation}
where $(b_{i,u},0\le u \le t)_{i=1}^3$ are three independent
one-dimensional Brownian bridges. Now, it is easy to check (see,
e.g., Yor \cite{yor}, Section 0.5) that if $\mathbb{W}^{(t)}$ is
the law of a one-dimensional bridge of duration $t$, then for
$s<t$,
\begin{equation}
\frac{d\mathbb{W}^{(t)}}{d\mathbb{W}}|_{\mathcal{F}_s}=
\left(\frac{t}{t-s}\right)\exp\left(-\frac{X_s^2}{2(t-s)}\right).
\end{equation}
Letting $t\to \infty$ and $s>0$ fixed, we see that the above
Radon-Nikodyn derivative converges to 1. This means that the
restrictions of $(b_{i,u},0\le u \le t)_{i=1}^3$ to
$\mathcal{F}_s$ converge to three independent Brownian motions. By
(\ref{exc-bes2}), it follows that the restriction of $(e_u,0\le u
\le t)$ converges to $(\sqrt{X_{1,u}^2+X_{2,u}^2+X_{3,u}^2},0\le u
\le s)$ in distribution, where $(X_{i,u},u\ge 0)_{i=1}^3$ are
three independent Brownian motions. The law of this process is, of
course, the same as $\P|_{\mathcal{F}_s}$, and hence Theorem
\ref{T:bdd} is proved.

\subsection{Slowly growing local time.}

We now consider a problem which may be considered the basic building block for the proof of Theorem \ref{T:pt}.
Let $f$ be a nonnegative nondecreasing function, and let
\begin{equation}\label{D:Eslow}
\mathcal{E}_t=\{L_t\le f(t)\}.
\end{equation}
Let
\begin{equation}
\mathbb{Q}_t(\cdot)=\bm(\cdot | \mathcal{E}_t)
\end{equation}
with $\{X_t, t \ge 0\}$ distributed according to the Wiener measure
$\mathbb W$. Note the difference between $\cE_t$ and $\cK_t$, where the conditioning concerns the entire growth of the local time profile up to time $t$, whereas
$\cE_t$ concerns only the value of $L$ at time $t$.

\begin{theorem} \label{T:slow}
Assume that $f(t) \to \infty$ as $t\to \infty$. Then $\mathbb{Q}_t$
converges weakly to the standard Wiener measure $\mathbb W$ on $\cC$.
\end{theorem}

This result may be a little surprising at first: even if $f$ grows as
slowly as $\log \log (t)$ we still obtain a recurrent process, indeed
a Brownian motion in the limit. What is going on is that the effect
of the conditioning is to create one very long excursion, but whose
starting point escapes to infinity as $t\to \infty$. As a result the
conditioning becomes trivial in the weak limit.

\begin{proof}
The strategy for the proof of Theorem \ref{T:slow} is similar to
that used in the proof of Theorem \ref{T:bdd}, and we will thus
give fewer details. We start again by noticing that the event
$\cE_t = \{L_t \le f(t)\}$ is equivalent to the event
$\{\tau(f(t))> t\}$. If we define the event $\cE'_t$ that there is
at least one excursion of duration greater than $t$ by time $\tau(f(t))$,
then we have again $\cE'_t \subset \cE_t$, and letting $\lambda = f(t) \nu(\zeta>t) =  f(t) \sqrt{2/(\pi t)}$,
\begin{equation}
{\mathbb  W}(\cE'_t)  = 1- e^{- \lambda } \sim 
%\sim {\mathbb W} (\cE_t)\sim f(t)\nu\big(\zeta > t\big) = 
f(t)\sqrt{ \frac 2{\pi t}}
\label{UBloctime}
 \end{equation}
 while once again, by L\'evy's identity and the reflection principle,
 \begin{equation}\label{UBloctime2}
 \mathbb{W}(\cE_t) = \mathbb{W}(|X_t| \le f(t) ) = \frac1{\sqrt{2\pi t}}\int_{ |x| \le f(t) } e^{ - ^x/ (2t)}dx \sim f(t)\sqrt{ \frac 2{\pi t}}
 \end{equation} 
Therefore, this time again, if $A \in \cF_\infty$ then $\mathbb W(A | \cE_t) \to \P(A)$ as $t \to \infty$ if and only if
$\mathbb{W} (A| \cE'_t) \to \P(A)$. Thus let $M = \sum_{i\ge 1} \delta_{(\ell_i, e_i)}$ be It\^o's excursion point process, and let $\tilde M$ be a realisation of $M$ conditionally given $\cE'_t$. Then reasoning as in \eqref{tildeM}, we may write
\begin{equation}
  \tilde M  = M'|_{(0, f(t)]} + \delta_{(U f(t), e)} + M|_{(f(t), \infty) \times \cW}
\end{equation}
where $M'$ is an independent realisation of $M$, $U$ is an independent uniform random variable on $(0,1)$, and $e$ has the distribution \eqref{e}. Letting $\tilde X$ be the path reconstructed from the points of $\tilde M$, we see that the distribution of $\tilde X$ up to time $\tau(Uf(t))$ is that of a standard Brownian motion. (We also have that $\tilde X_t=e(t-\tau(Uf(t)))$ for
$\tau(Uf(t))\le t \le \tau(Uf(t)) + \zeta(e) $, and that the process $e$ still converges to a
3-dimensional Bessel process as $t\to \infty$ with a random sign, but as we will see this is irrelevant). Observe that now the random
variable $\tau(Uf(t))$ tends to infinity almost surely since $f(t) \to \infty$ as $t \to \infty$. In particular this convergence holds in probability. Fix any $s \ge 0$ and $A \in \cF_s$. Then
$$
|\P((\tilde X_u, 0\le u \le s) \in A ) - \mathbb{W}(A)| \le 2 \P(\tau(Uf(t))>s) \to 0
$$
as $t \to \infty$. This proves that the law of $(X_u , 0 \le u \le s)$ conditionally given $\cE'_t$, converges weakly as $t \to \infty$ (or indeed in total variation) towards $\mathbb{W}|_{\cF_s}$. Therefore, by \eqref{UBloctime} and \eqref{UBloctime2}, the law of $(X_u,0\le u \le s)$
given $\cE_t$ also converges weakly towards $\mathbb{W}|_{\cF_s}$. This proves Theorem \ref{T:slow}. 
\end{proof}

In words, the
conditioning ``becomes invisible'' asymptotically: the effect of
conditioning on $\mathcal{E}_t$ is to add a three-dimensional
Bessel which starts far away in the future. Observe that, as a
byproduct, given $\mathcal{E}_t$,
\begin{equation}
\label{E:qty} \frac{L_t}{f(t)} \underset{t\to
\infty}\longrightarrow U
\end{equation}
in distribution, where $U$ is a uniform random variable on
$(0,1)$.

\section{Proof of the main result}
\label{S:proof}
\subsection{Preliminary estimates}

To begin with the proof of Theorem \ref{T:pt}, we first divide the positive half-line into dyadic blocks.
Let $t_j=2^j$, and let $I_j=[t_{j-1},t_j)$. We use as a shorthand notation $K_n = \cK_{t_n}$. Our strategy will consist in studying a less stringent constraint for which the analysis is easier.
Let $C'_n$ be the following modified constraint:
\begin{equation}
 \label{Cprdef}
 C'_n=\{L_s - L_{t_{n-1}} \le f(s), \ \ \text{ for all $s \in I_n$}\},
\end{equation}
and define analogously
\begin{equation}
 \label{Kprdef}
 K'_n = \bigcap_{j=1}^n C'_j
\end{equation}
Note that $K'_n$ is obviously a ``weaker" constraint, in the sense that $K_n \subset K'_n$. The difference between $K_n$ and $K'_n$ is that $K'_n$ is ``more Markovian" and hence easier to analyse. We will show however that $K_n$ has a positive probability to occur given $K'_n$ (bounded away from zero), which means that any result which is true for $K'_n$ with probability 1, will also be true for $K_n$. Note also that $K'_n$ may be realised as the event $\cK_{t_n}$ associated with a modified function $\hat f(t)$, which stays constant on each interval $I_j$ and taks the value $f(t_j)$ on this interval.

Let $H_j = \{ X_s = 0 \text{ for some  $s\in I_j$}\}$, and let $A'_j = H_j \cap C'_j$.

\def\cZ{\mathcal{Z}}
\def\cB{\mathcal{B}}

\medskip We start the proof with the following general result which we will extensively use as an upper-bound on the growth of $X$ given ${\cal K}_t$. For a function
$\omega \in \cC$, let $Z(\omega)=\{x\ge 0:\omega(x)=0\}$ be the
set of its zeros; we view $Z$ as a random variable under the
probability measure $\mathbb{W}$. That is, we equip the set $\Omega$ of all closed subsets of the real nonnegative half-line with the $\sigma$-field $\cB$ defined by $Z\in \cB$ if and only if $\{\omega \in \cC: Z(\omega) \in \cZ\}\in \cF $, where $\cF$ is the Borel $\sigma$-field on $\cC$ generated by the open sets associated with local uniform convergence. Let $\cal Z$ be a set of closed
subsets of $\mathbb{R}_+$, such that $\mathbb{W}(Z \in {\cal Z})>
0$. Let $\preceq$ be the order on $\Omega$ defined by $\omega
\preceq  \omega'$ if and only if $\omega(x) \le \omega'(x)$ for
all $x\ge 0$. Finally, let $\bess$ denote the law of a three-dimensional Bessel process, which is obtained when one considers, e.g., the Euclidean
norm of a Brownian motion in $\mathbb{R}^3$.

\begin{lemma} \label{L:subdiff}
Conditionally on $\{Z\in {\cal Z}\}$, the $|X|$ is
dominated by a three-dimensional Bessel process. More precisely,
for any continuous bounded functional $F$ on $\cC$, which is
nondecreasing for the order $\preceq$,
\begin{equation}
\label{E:subdiff} \mathbb{W}(F(|X_s|,s\ge 0)|Z\in {\cal Z}) \le
\bess(F(X_s,s\ge 0)),
\end{equation}
where $\bess$ denote the law of a 3-dimensional Bessel process.
\end{lemma}

\begin{proof}
Let $(R_s,s\ge 0)$ be a three-dimensional Bessel
process. Recall that $(R_s,s\ge 0)$ is solution of the stochastic
differential equation:
\begin{equation}\label{bessel}
 dR_s= dB_s + \frac1{R_s}ds.
\end{equation}
We work conditionally on $\{Z(\omega)=z\}$ for a given $z\in
\mathcal{Z}$. $z$ being closed, its complement is open and defines
open intervals which are referred to as excursion intervals. Given
$\{Z=z\}$, the law of $X$ can be described as a concatenation of
independent processes whose laws are precisely It\^o excursions
conditioned on their duration $\zeta$, where $\zeta$ is the length
of the excursion interval of $z$ under consideration.
More formally, for $\zeta>0$, let $n_\zeta$ denote the law of a brownian excursion conditioned to have duration $\zeta$. Note that this \emph{a priori} only makes sense Lebesgue-almost everywhere (see, e.g., V.10 in Feller \cite{feller2}). However, the Brownian scaling property implies that $n_\zeta$ is weakly continuous for the topology induced by local uniform convergence, so that $n_\zeta$ is defined unambiguously for all $\zeta>0$.
Furthermore, it is not hard to show (see, e.g., Pitman \cite{pitman-sde}, formula (28)),
that under $n_1$, $X$ is a solution to the stochastic differential equation:
\begin{equation}
dX_s= dB_s +\frac1{X_s}ds - \frac{X_s}{1-s}ds,
\end{equation}
for which there is strong uniqueness.
By Brownian scaling, under $n_\zeta$, $X$ is thus
the unique in law solution to the stochastic differential equation:
\begin{equation}\label{sdeexc}
dX_s = dB_s +\frac1{X_s}ds - \frac{X_s}{\zeta-s}ds
\end{equation}
Start with a small parameter $\delta>0$. If $X \in \cC$ is a continuous function, let $X^{(\delta)}$ denote the element of $\cC$ obtained by removing all the excursions of $X$ of duration smaller than $\delta$ (alternatively, retaining all the excursions of length greater than $\delta$). Since there are always no more than a finite number of excursions longer than $\delta$ on any compact interval, there is no problem in ordering these excursions chronologically. Likewise, let $X_{(\delta)}$ denote the process $X$ where all the excursions longer than $\delta$ have been removed. Then It\^o's excursion theorem tells us that under $\bm$, the processes $X^{(\delta)}$ and $X_{(\delta)}$ are independent. Furthermore, the law of $X^{(\delta)}$ can be obtained by concatenating an i.i.d. sequence of excursions conditioned to have length greater than $\delta$. As a consequence, let $Z^{(\delta)}$ be the zero set of $X^{(\delta)}$, and for $z\in \Omega$ let $z^{(\delta)}\in \Omega$ be the closed subset of $\R_+$ where all intervals of $z^c$ of length smaller than $\delta$. Then, given $\{Z(\omega) = z\}$, $X^{(\delta)}$ may be described as a concatenation of independent excursions of respective laws $n_{\zeta_1}, \ldots, $ together with independent random signs, where $\zeta_1, \ldots$ are the chronologically ordered interval lengths of $(z^{(\delta)})^c$.
Since $\bm(Z\in \cZ)>0$, it follows that
\begin{align}
 E[F(|X^{(\delta)}|) | Z\in \cZ] & = \frac1{\bm(Z\in \cZ)} \int_\Omega \indic{z\in \cZ}\mu(dz) E[F(|X^\delta|) | Z=z]\label{condzero}
\end{align}
By the above two observations, given $\{Z=z\}$, $X^{(\delta)}$ may be obtained as a concatenation of independent solutions to the stochastic differential equations (\ref{sdeexc}). We may thus construct a coupling of $X^{(\delta)}$ with a Bessel process $(R_s, s\ge 0)$ as follows. Fix a Brownian motion $(B_s,s\ge 0)$. Let $\zeta_1, \zeta_2,\ldots$ denote the lengths of the excursions of $z^{(\delta)}$, ordered chronologically. We first construct
$(X^{(\delta)}_s,0\le s\le \zeta_1)$ and $(R_s,0\le s\le \zeta_1)$ on the same probability space by solving (\ref{bessel}) and (\ref{sdeexc}) using the same Brownian motion $(B_s, 0 \le s\le \zeta_1)$ (this is possible by strong uniqueness), and by giving $X^{(\delta)}$ a random sign on that interval. Since (\ref{sdeexc}) contains an additional negative drift term compared to (\ref{bessel}), it follows that
\begin{equation}\label{comparison}
|X^{(\delta)}_s| \le R_s
\end{equation}
 holds pointwise on $[0,\zeta_1]$. By the Markov property of Brownian motion at time $\zeta_1$, there is no problem in extending this construction on the interval $[\zeta_1, \zeta_1+\zeta_2]$, this time using the Brownian motion $(B_s - B_{\zeta_1}, \zeta_1 \le s \le \zeta_1  +\zeta_2)$ for both (\ref{bessel}) and (\ref{sdeexc}). Since $R_{\zeta_1}\ge |X^{(\delta)}_{\zeta_1}|$ by (\ref{comparison}), and since (\ref{sdeexc}) still contains an additional negative drift term compared to (\ref{bessel}), we conclude that the comparison (\ref{comparison}) still holds true on the interval $[\zeta_1, \zeta_1 + \zeta_2]$. It is trivial to extend this construction on all of $\R_+$ by induction, and to obtain (\ref{comparison}) pointwise on all of $\R_+$. Since $F$ is nondecreasing, we deduce that
 $$
 E[F(|X^\delta|) | Z=z]\le E^{\text{(Bess)}}[F(X)].
 $$
Plugging this into (\ref{condzero}), we get
\begin{align}
 E[F(|X^{(\delta)}|) | Z\in \cZ] & \le E^{\text{(Bess)}}[F(X)] \label{condzero1}
\end{align}
On the other hand, $F$ is by assumption continuous for the local uniform convergence, and it is easy to see that $X^{(\delta)}$ converges almost surely to $X$ in the local uniform convergence. Since $F$ is furthermore bounded, we conclude by the dominated convergence theorem that
$$
E[F(|X|) | Z\in \cZ]  \le E^{\text{(Bess)}}[F(X)]
$$
as requested.
\end{proof}

Lemma \ref{L:subdiff} has the following concrete and useful consequence. Observe
that the event $K_n$ may be written as $\{Z(\omega)\in {\cal
Z}_n\}$ for some set ${\cal Z}_n$. Indeed, recall that almost
surely for every $t\ge 0$:
$$
L_t = \lim_{\delta \to 0} \left(\frac\pi 2\delta\right)^{1/2} N^\delta_t
$$
where $N_t^\delta $ is the number of excursions longer than $\delta$ by time $t$. (See, e.g., Proposition (2.9) of chapter XII in \cite{revuz-yor} for a proof.) Thus we obtain in particular:
\begin{equation}\label{E:subdiff2}
\mathbb{W}(F(|X_s|,s\ge 0) |K'_n) \le \bess (F(X_s,s\ge 0)).
\end{equation}

\begin{rmk}
An easy modification of the above proof shows that if $X_0=x\neq 0$ almost surely (i.e., if $\bm$ is replaced by the law of $(x+B_t,t\ge 0)$ in (\ref{E:subdiff})), then the same result holds where $R$ is a Bessel process also started at $x$.
\end{rmk}

Corresponding to this upper-bound on the growth of $X$ given $\cK_t$, we will prove a matching lower-bound and show that given $\cK_t$, $|X|$ dominates a reflecting Brownian motion.
in terms of reflecting Brownian motion.

\begin{lemma} \label{L:supd}
For all $T\ge 2$,
$$
\mathbb{W}_T(F(|X_s|,s\ge 0)) \le \bm (F(|X_s|,s\ge 0)).
$$
\end{lemma}

\begin{proof} It is well known that a conditioned Brownian motion can be seen as an $h$-transform of the process and hence as adding a drift to the Brownian motion. It suffices to show that this drift is always the same sign as the current position of the process. Formally, fix $T\ge 0$ and let
\begin{equation}
\cA=\{\omega\in \Omega: L_s(\omega) \le f(s) \text{ for all } 0\le s \le T \}.
\end{equation}
For $t\ge t_0$, and $0\le \ell\le f(t)0$, define:
\begin{equation}
\cA(t,\ell)=\{\omega \in \Omega: L_s \le f(s+t) \text{ for all } 0 \le s \le T-t\}.
\end{equation}
That is, $\cA$ is the initial constraint and $\cA(t,\ell)$ is what remains of that constraint after $t$ units of time and having already accumulated of local time at zero of $\ell$ by that time. Let
\begin{equation}
h(x,t,\ell) = \bm_x(\cA(t,\ell)).
\end{equation}
Then the conditioned process $\mathbb{P}_T$ may be described by using Girsanov's theorem, to get that under $\mathbb{P}_T$, then $X$ is a solution to
\begin{equation}
dX_t = dW_t + \nabla_x \log h(X_t,t,L_t)
\end{equation}
where $L$ is the local time at the origin of $X$ and $W$ is a one-dimensional Brownian motion. Details can be found for instance in \cite[IV.39]{rw} in the case where the constraint depends only on the position of $X$ (and not on its local time as well) but the proof remains unchanged in this case as well. Assume to simplify that $x\ge 0$, and let us show that $\nabla_x \log h(x,t,\ell) \ge 0$ for all $t\ge t_0$ and $0\le \ell \le f(t)$. It suffices to prove that $(\partial h/ \partial x)(x,t, \ell)\ge 0$. Thus it suffices to prove that for all $y\ge x$ sufficiently close to $x$, \begin{equation}\label{mono}
h(y,t,\ell) \ge h(x,t,\ell).
\end{equation}
In other words, we want to show that the probability of $\cA(t,\ell)$ is monotone in the starting point. We use a coupling argument similar to the one we used in \cite[Lemma 8]{bb}. Let $X$ be a standard Brownian motion started at $x$ and let $Y$ be another Brownian motion, started
at $y$. Consider the stopping time:
\begin{equation}\label{E:couplT}
\tau=\inf\{t>0:Y_t=|X_t|\}
\end{equation}
and construct the process $Z$ defined by $Z_t=Y_t$ for $t\le \tau$
and for $t>\tau$:
$$
Z_t=\begin{cases}
X_t &\text{ if} X_\tau=Y_\tau\\
-X_t & \text{ if} X_\tau=-Y_\tau
\end{cases}
$$
Then $Z_t$ has the law of Brownian motion started at $y$.
Moreover, we claim that for any $s>0$, in this coupling:
$$
L_s(Z)\le L_s(X).
$$
Indeed, note that for any $s\le \tau$ we have $L_s(Z)=0$ since $Z$
cannot hit 0 before $\tau$. Afterwards, the local time of $Z$
increases exactly as that of $X$, hence in general, for all $s\ge 0$:
\begin{equation}
\label{E:couplL} L_s(Z)=(L_s(X)-L_\tau(X))_+.
\end{equation}
From (\ref{E:couplL}), we also see that the increment of $L_s(Z)$
over any time-interval is smaller or equal to the increment of
$L_s(X)$ over the same time-interval. Therefore, if $X$ satisfies
$\cA(t,\ell)$, then so does $Z$. As a consequence,
$$
\bm_x(\cA(t,\ell))\le \bm_y(\cA(t,\ell))
$$
which is the same as (\ref{mono}). Thus Lemma \ref{L:supd} is proved.
\end{proof}

A first consequence of these two dominations is a simple proof that the conditioned processes form a tight family of random paths.

\begin{lemma}
\label{L:tight} The family of measures $\{\mathbb{P}_t\}_{t\ge 0}$ is tight.
\end{lemma}

\begin{proof}
By classical weak convergence arguments (see, e.g., Billingsley \cite{billingsley}), it suffices to prove that for each fixed $A$, the following two limit relations
hold:
\begin{equation}\label{22}
\lim_{b \to \infty} \limsup_{T \to \infty} \mathbb{P}_T\left(\sup_{0\le s \le  A}
|X(s)| > b\right) = 0, %\teq(22)
\end{equation}
and for each $\eta > 0$
\begin{equation}\label{23}
\lim_{\ep \downarrow 0} \limsup_{T \to \infty}
\mathbb{P}_T\left(\sup_{0\le  s,s' \le A, |s-s'|< \ep}|X(s)-X(s')|  >
\eta\right) = 0. %\teq(23)
\end{equation}
(\ref{22}) is a direct consequence of Lemma \ref{L:subdiff}, so we move to the proof of (\ref{23}). Basically, the domination above and below by a Bessel process and a reflecting Brownian motion give us a uniform Kolmogorov estimate:
$$
\mathbb{E}_T(|X_t - X_s|^4 ) \le C(t-s)^2
$$
for some $C>0$ and all $t,s>0$. For instance, to get a bound on $\mathbb{E}((X_t - X_s)_+^4)$, it suffices to note that if $s<t$ without loss of generality, then conditionally on $\cF_s$, the process $(X_{s+u},u\ge 0)$ under $\bm_T$ may be realised again as a Brownian motion conditioned upon satisfying a constraint of the form $\cK_{T-s}$, and hence the domination by a 3-dimensional Bessel process started from $X_s$ holds. The bound on $\mathbb{E}((X_t - X_s)_-^4)$ follows similarly.

Thus, if $n$ is an integer and $0\le k \le A2^n$, by Markov's inequality:
$$
\bm_T\left(\sup_{0\le k\le A2^n}|X_{(k+1)2^{-n}} - X_{k2^{-n}}| \ge 2^{-n/8}\right) \le A 2^n C 2^{n/4 -2n} = C A2^{-3n/4}.
$$
Thus, summing over $n\ge N$ for some $N\ge 1$:
\begin{equation}\label{kolm1}
\bm_T \left(\exists \ n \ge N, \sup_{0\le k\le A2^n}|X_{(k+1)2^{-n}} - X_{k2^{-n}}| \ge 2^{-n/8}\right) \le C A2^{-3N/4}.
\end{equation}
Now, let $s,t$ be two dyadic rationals such that $s<t$ and $|s-t|<2^{-N}$, and assume that the complement of the event in the left-hand side of (\ref{kolm1}) holds. Let $r\ge N$ be the least integer such that
$|s-t|> 2^{-r-1}$. Then there exists an integer $0\le k \le A2^r$, as well as two integers $\ell,m$ such that
\begin{align*}
s&= k2^{-r} - \eps_1 2^{-r-1} - \ldots - \eps_\ell 2^{-r-\ell}; \\
t& = k2^{-r} + \eps'_1 2^{-r-1}+\dots + \eps_m2^{-r-m};
\end{align*}
with $\eps_i,\eps'_i\in \{0,1\}$. For $0\le i \le \ell$ and $0\le j \le m$ let
\begin{align*}
s_i&= k2^{-r} - \eps_1 2^{-r-1} - \ldots - \eps_\ell 2^{-r-i}; \\
t_j& = k2^{-r} + \eps'_1 2^{-r-1}+\dots + \eps_m2^{-r-j}.
\end{align*}
Thus by the triangular inequality:
\begin{align*}
|X_t - X_s| &\le |X_{t_0} - X_{s_0}| + \sum_{i=1}^\ell  |X_{s_i} - X_{s_{i-1}}| + \sum_{j=1}^m |X_{t_j} - X_{t_{j-1}}|\\
&\le  2^{-r/8} + \sum_{i=1}^\ell 2^{-(r+i)/8} + \sum_{j=1}^m 2^{-(r+j)/8}\\
&\le \frac3{1-2^{-1/8}} 2^{-r/8}\le c|t-s|^{1/8}\\
\end{align*}
with $c= 3/(1-2^{-1/8}) 2^{1/8}$. Therefore, by the density of dyadic rationals in $[0,A]$, it follows that
$$
\bm_T\left(\sup_{s,t \in [0,A]; |s-t|\le \epsilon} |X_s - X_t|> c|s-t|^{1/8}\right)\le CA \eps^{3/4}.
$$
where $\eps = 2^{-N}$. Since the right-hand side does not depend on $T$, we may take the limsup as $T\to \infty$, and (\ref{23}) follows directly.
\end{proof}

\subsection{Proof of the Theorem \ref{T:pt}.}

%Fix $B$ an arbitrary event in $\cB_{j-1}$.

\begin{lemma}\label{L:Ajub}
Let $\gamma>0$.
There exists $C>0$ such that for all $n$ large enough,
\begin{equation}\label{Ajub}
\mathbb{W}(A'_j|K'_{j-1}) \le C\frac{f(t_j)}{\sqrt{t_j}}.
\end{equation}
\end{lemma}

\begin{proof}

We use an \emph{a priori} rough
upper-bound, based on the local time accumulated at the end of the
interval $I_j$.
Let $\tau_j = \inf \{t>t_{j-1}:
X_t=0\}$.
Conditionally on $H_j$ and $\cF_{\tau_j}$, $(X_{\tau_j + t},t\ge 0)$ is a Brownian motion started at 0 by the strong Markov property.
Thus by (\ref{UBloctime}),
$$
\bm( C'_j | \cF_{\tau_j}; K'_{j-1}, H_j) \le C \frac{f(t_j)}{\sqrt{t_j - \tau_j}}.
$$
Now, note that if $T^0$ denote the hitting time of zero and $\bm^x$ denote the Wiener measure started from $x$ and $S_t = \sup_{s\le t } X_s$, then by translation invariance and the reflection principle, we have  for all $t>0$ and $x\neq 0$,
\begin{eqnarray}
\frac{\mathbb{W}^x(T^0\in dt)}{dt} &=& \frac{d}{dt} \bm(S_t>x) \nonumber\\
&=&\frac{d}{dt}2\int_{x/\sqrt{t}}^\infty e^{-v^2/2} \frac{dv}{\sqrt{2\pi}}\nonumber\\
&=&ct^{-3/2}xe^{-x^2/2t}\nonumber\\
&\le& ct^{-1}\label{E:hit}
\end{eqnarray}
where $c>0$ is a universal constant independent of $t$ and $x$.
As
a consequence, conditioning on $\cF_{t_{j-1}}$ and letting $x= X_{t_{j-1}}$,
\begin{align}
\mathbb{W}(A'_j| \cF_{t_{j-1}}, X_{t_{j-1}}=x)&\le
\int_{t_{j-1}}^{t_j} \mathbb{W}(\tau_j \in ds | \cF_{t_{j-1}}, X_{t_{j-1}} = x) \frac{Cf(t_j)}{ \sqrt{t_j - s}}  \nonumber \\
&\le Cf(t_j) \int_{t_{j-1}}^{t_j} \frac1{\sqrt{t_j-s}} \frac{ds}s \nonumber \\
&\le C\frac{f(t_j)}{\sqrt{t_{j-1}}} \int_1^2 \frac1{\sqrt{2-u}}\frac{du}u \le  C\frac{f(t_j)}{\sqrt{t_j}}\label{E:hit2}
\end{align}
Taking the expectation of both sides, we deduce the result.
\end{proof}

We now show how the above lemma can be extended to show that the probability of hitting zero during the interval $I_j$ is uniformly small as $n\to \infty$. This is the key lemma of the proof, since it precisely deals with the way the self-interaction of the process ``propagates" down from $+\infty$ to the finite window $[t_{j-1}, t_j]$ we are considering .

\begin{lemma}\label{L:Ajub2}
There exists $C>0$ and $j_0\ge 1$ such that
\begin{equation}\label{Ajub2}
\mathbb{W}(H_j|K'_n) \le C\frac{f(t_j)}{\sqrt{t_j}},
\end{equation}
uniformly in $n\ge j+1$ and $j\ge 1$.
\end{lemma}

\begin{proof}
We first note that for all $n\ge j+1$,
\begin{align}
\mathbb{W}(H_j |K'_n) %&= \frac{\bm(A_j\cap K_n)}{\bm(K_n)}\nonumber \\
%&= \frac{\bm(A_j \cap K_{j-1} )}{\bm(K_j \cap A)} \frac{\bm(A\cap K_j)}{\bm(K_j)}\frac{\bm(K_j)}{\bm(K_n)}\nonumber \\
&= \bm(A'_j| K'_{j-1}) \frac{\bm(K'_n | K'_j \cap H_j)}{\bm(K'_n | K'_j)} \label{bas0}.
\end{align}
In view of Lemma \ref{L:Ajub}, it thus suffices to prove that
\begin{equation}
 \label{bas1}
 \bm(K'_n | K'_j \cap H_j) \le C \bm(K'_n | K'_j)
\end{equation}
for all $j\ge j_0$ and $n\ge j+1$, for some uniform $j_0,C>0$.
Consider a process $(X_1(t),t\ge 0)$, (resp. $(X_2(t),t\ge 0)$), which is a Brownian
motion conditioned on $K'_j\cap H_j \cap \{|X_1(t_j)| \le \sqrt{t_j}\}$ (resp. $K'_j  \cap \{|X_2(t_j)| > \sqrt{t_j}\}$). Assume further that $X_1$ and $X_2$ are independent. Note that conditionally on $X_1(t_j)$ and $X_2(t_j)$, both processes evolve after time $t_j$ as independent Brownian motions started from their respective positions at this time. Construct a process $\hat X$ as follows. Let $\tau:= \inf\{t>t_j: |X_2(t)| =|X_1(t)|\}$, and let $\epsilon = X_2(t) /X_1(t) \in \{-1,1\}$.
Define
$$
\hat X_t = \begin{cases}
X_2(t) &\text{ for } t\le \tau \\
\epsilon X_1(t) &\text{ for } t\ge \tau.
\end{cases}
$$
Then by the Markov property, $\hat X$ has the same distribution as $X_2$. Moreover, note that $|X_1(t_j)| \le |X_2(t_j)|$ holds almost surely, and hence
$$
|X_1(t)| \le |\hat X(t)|,
$$
for all $t\ge t_j$. Recall that by \eqref{loctime}, if $(B_u, u \ge 0)$ is a Brownian motion, then for all $s\le t$, $L_t(B) - L_s(B) = \lim_{\eps \to 0} \frac1{2\eps}\int_s^t\indic{|B_u| \le \eps} du  $ almost surely. It follows immediately from this and the above that
$$
L_t(\hat X) - L_{t_j}(\hat X) \le L_t( X_1) - L_{t_j}(X_1), \ \ t\ge t_j.
$$
Therefore, if $X_1 \in K'_n$ then automatically $\hat X \in K'_n$. Hence we deduce that
\begin{equation}\label{comparison1}
P(X_1 \in K'_n) \le P(\hat X \in K'_n).
\end{equation}
Now, observe that
\begin{align}
P(\hat X \in K'_n) & = \bm( K'_n | K'_j  \cap \{|X(t_j)| > \sqrt{t_j}\}) \nonumber \\
&\le  \frac{\bm( K'_n ; K'_j  )}{\bm(K'_j) \bm(|X(t_j)| > \sqrt{t_j} |K'_j) } \nonumber \\
&\le \frac1{p_1}\bm(K'_n |K'_j) \label{comp2}
\end{align}
where $ \bm(|X_1(t_j)| > \sqrt{t_j} |K'_j) \ge p_1>0$ does not depend on $n$, by Lemma \ref{L:supd}.

Similarly,
\begin{align}
 P(X_1 \in K'_n) & = \bm( K'_n | K'_j \cap H_j \cap \{|X(t_j)| \le \sqrt{t_j}\}) \nonumber \\
 &\ge \frac{\bm(K'_n;K'_j \cap H_j; |X(t_j)| \le \sqrt{t_j} )}{\bm(K'_j \cap H_j)} \nonumber\\
 &\ge \bm(K'_n |K'_j\cap H_j) \bm(|X(t_j)| \le \sqrt{t_j} |K'_j \cap H_j) \nonumber \\
 &\ge p_2 \bm(K'_n |K'_j\cap H_j),\label{comp3}
\end{align}
where $\bm(|X(t_j)| \le \sqrt{t_j} |K'_j \cap H_j)\ge p_2>0$ does not depend on $n$ by Lemma \ref{L:subdiff}. Putting together (\ref{comparison1}), (\ref{comp2}), and (\ref{comp3}), we find that
$$
\bm(K'_n|K'_j\cap H_j) \le \frac1{p_1p_2} \bm(K'_n |K'_j)
$$
which is precisely what was required.
\end{proof}

\begin{lemma}
 \label{L:probKn}
Assume that $\gamma>1$. There exists $c_1, c_2>0$ which does not depend on $n$ such that
 \begin{equation}\label{probKn}
 c_1 t_n^{-1/2} \le \bm(K_n) \le \bm(K'_n) \le c_2 t_n^{-1/2},
 \end{equation}
 for all $n$ sufficiently large.
\end{lemma}

\begin{proof}
The first inequality is a trivial consequence of the observation that $K_n$ occurs automatically if $\tau_1 >t_n$ and $L_{t_1} < f(t_1)$, where $\tau_j= \inf\{t\ge t_j: X_t = 0\}$. The second inequality is also trivial since $K_n \subset K'_n$. It thus remains to show the third inequality.
  Let $G_{k,n} = \bigcap_{j=k}^n H_j^c = \{\tau_j > t_n\}$. Then by Lemma \ref{L:Ajub2}, letting $g(t)  =f(t)/\sqrt{t}$, which is non-increasing by assumption, we have that
  $$
  \bm(G^c_{k,n} |K'_n) \le C \sum_{j=k}^n g(t_j) \le C\sum_{j=k}^\infty g(t_j).
  $$
  Since $g$ is non-increasing, note that
  $$
  \int_{j-1}^j g(2^t)dt \ge g(t_j)
  $$
  and thus
  $$
  \sum_{j=k}^\infty g(t_j)\le \int_{k-1}^\infty g(2^t) dt = \frac1{\ln 2}\int_{t_{k-1}}^\infty \frac{g(u)}u du = \frac1{\ln 2}\int_{t_{k-1}}^\infty \frac{f(u)}{u^{3/2}} du.
  $$
Since $f$ satisfies the condition (\ref{cond:f}), then the integral in the right-hand side is finite and we can find $J$ large enough that $\bm(G^c_{k,n} |K'_n)\le 1/2$ when $k\ge J$. Thus we get,
\begin{align*}
\bm(K'_n) & = \bm(K'_n ; \tau_J > t_n) + \bm( K'_n ; \tau_J \le t_n) \\
&\le \bm ( \tau_J> t_n) + \frac12 \bm(K'_n).
\end{align*}
Thus
\begin{equation}\label{tauJ}
\frac12 \bm(K'_n) \le \bm ( \tau_J> t_n) \sim c_2 t_n^{-1/2}
\end{equation}
as $n \to \infty$, for some $c_2>0$ that does not depend on $n$. The proof follows immediately.
\end{proof}

As a consequence of Lemma \ref{L:probKn}, we see that $\bm(K_n |K'_n) \ge c_1/c_2>0$, uniformly in $n$. In particular, applying Lemma \ref{L:Ajub2}, we get the following estimate valid for all $t\ge t_j$:
\begin{equation}\label{HjKn}
\bm(H_j |\cK_t) \le C' g(t_j)
\end{equation}
for some $C'>0$ which does not depend on $t$. It is now easy to conclude that any weak limit of $\bm_t $ is transient:

\emph{Proof of Theorem \ref{T:pt}}. Fix $\eps>0$.
Observe that by (\ref{HjKn}), we can find $J\ge 1$ such that $\bm(\tau_J >t |K_n)\ge 1-\eps$ for all $t\ge t_j$.
Let $A>0$ be an arbitrarily large number, and fix $u <v$. Let
$$
R=\left\{ \inf_{s\in [t_J+u,t_J + v]} |X_s| \le A\right\}.
$$
\begin{align*}
\bm(\cK_t;R) & = \bm(\cK_t;R; \tau_J < t) + \bm( K_n;R, \tau_J\ge t)\\
& \le \bm(\cK_t) \eps + \bm(\tau_J>t)\bm(R|\tau_J>t)
\end{align*}
Now, by Lemma \ref{L:probKn} we see that dividing by $\bm(\cK_t)$ we find
$$
\bm(R|\cK_t) \le \eps + c_3 \bm(R|\tau_J>t).
$$
Now, observe that given $\tau_J> t$, and given $\cF_{t_j}$, the process $\{X(t),t\ge t_j\}$ converges weakly to a 3-dimensional Bessel process started from $X(t_j)$ as $t\to \infty$. Note that $f\mapsto \inf_{[u,v]}|f|$ is a continuous functional on the space $\cC$ of continuous sample paths (equipped with the topology of local uniform convergence). Since a Bessel process started from $x>0$ always dominates a Bessel process started at 0, we deduce that
$$
\limsup_{t\to \infty} \bm(R|\tau_J>t) \le \bess(R')
$$
where $\bess$ denotes the law of a 3-dimensional Bessel process started at 0, and
$$
R'=\left\{ \inf_{s\in [u,v]} |X_s| \le A\right\}.
$$
Since $\bess$ is almost surely transient, it follows that we can find $u$ large enough such that \emph{for all} fixed $v>u$, $\bess(R')\le \eps$ (independently of $v$), in which case we obtain
$$
\limsup_{t\to \infty} \bm(R|\cK_t) \le \eps(1+c_3).
$$
Using once again the fact that the infimum over a compact set is a continuous functional, we see that for any weak subsequential limit $\mathbb{P}$ of $\bm(\cdot|\cK_t)$,
$$
\mathbb{P}(R) \le \eps(1+c_3),
$$
for all $v>0$ arbitrarily large.  Letting $v\to \infty$ and changing $\eps$ into $\eps/(1+c_3)$, we obtain by monotone convergence:
$$
\mathbb{P}(\inf_{[u,\infty)} |X_s| \le A) \le \eps.
$$
Thus if $\Lambda = \sup \{t\ge 0: |X_t|\le A\}$, we have
$\mathbb{P}(\Lambda>u) \le \eps$. Since $\eps>0$ is arbitrary, we have proved that $\Lambda<\infty$, $\mathbb{P}$-almost surely. Since $A$ is arbitrary, this proves that $\mathbb{P}$ is transient almost surely.
\qed

\section{Brownian motion with bounded negative part}
\label{S:neg}
Until now, we have only considered conditionings which involve the
local time of a Brownian motion at a specified point. The next
result studies the case where the
forbidden region is a semi-infinite interval. Let $(X_t,t\ge 0)$
be a one-dimensional Brownian motion and let $A_t$ be the
additive functional of $X$ defined by
\begin{equation}
A_t=\int_0^t \indic{X_s<0}ds.
\end{equation}
$A_t$ is known as the negative part of Brownian motion, and is nothing else than the time spent by $X$ in the negative half-axis. Let
$$
\mathcal{E}_t=\{A_t\le 1\}
$$
and let $\mathbb{Q}_t$ be the measure defined by conditioning the
Wiener measure on the event $\mathcal{E}_t$.

\begin{theorem} \label{A-}
As $t\to \infty$, $\mathbb{W}_t$ converges weakly to a measure
$\mathbb{P}$ which is transient almost surely. Moreover, under $\mathbb{P}$ we have
$$
A_\infty \overset{d}= \mathbf{U}^2,
$$ where $\mathbf{U}$ is a uniform random variable on $(0,1)$.
\end{theorem}

An explicit description of the measure $\mathbb{Q}$ of the
limiting process is given in the proof. In a way that is analogous
to Theorem \ref{T:bdd}, the process is also made up of several
independent pieces glued together at a certain random time. The
proof of Theorem \ref{A-} uses explicit descriptions of the
Brownian path and precise distributional results (such as Paul
L\'evy's arcsine law). At this point it is not clear how to extend
this result to time-dependent conditionings in the manner of
Theorem \ref{T:slow} or Theorem \ref{T:pt}.

\begin{proof}
We start by recalling Paul L\'evy's second arcsine law, which states that
$A_t=_dg_t$, where
$$g_t=\sup\{s\le t: X_s=0\}$$
and both of these random variables have the arcsine law:
\begin{equation}\label{arcsine}
\bm(\frac{A_t}t \in dx)=\frac{1}{\pi
\sqrt{x(1-x)}}\indic{x\in(0,1)}dx.
\end{equation}
(See, e.g., Theorem 2.7 in Chapter VI, and (3.20) in Chapter III of \cite{revuz-yor}.)
It is also well-known (see, for instance, Theorem 4.1 of \cite{yor}) that there is
the following decomposition of the sample path of Brownian motion
at $g_t$.

\begin{itemize}
\item $(B_s,s\le g_t)$ and $(B_{g_t +s},s\ge 0)$ are independent.

\item $\ds (\frac1{\sqrt{g_t}}B_{sg_t},0\le s\le 1)$ is a standard
Brownian bridge of duration 1.
\end{itemize}

Therefore, by independence between the two pieces, the event that $\{g_t\le 1\}$ and the excursion containing
1 is positive has probability:
$$
\bm(g_t\le 1 \text{ and } X_1>0) =(1/2)\bm(A_t\le 1)
$$
It follows that:
\begin{equation}
\bm(g_t\le 1 \text{ and } X_1>0 |A_t\le 1)=1/2
\end{equation}
On this event it is clear that $X$ is transient, so transience
of $W$ occurs with probability at least $1/2$. In fact a more
precise argument allows one to to generalize this and get
transience with probability 1. We will show that there is a
constant that for every $A>1$
\begin{equation} \label{g_t > A}
 \lim_{t\to \infty} \bm(g_t>A|\mathcal{E}_t) = 1/(2\sqrt{A})
\end{equation}
where by definition $\mathcal{E}_t=\{A_t\le 1\}$ is the event
that we condition on. To see why (\ref{g_t > A}) holds, we start
by remarking that due to the Arscine law:
\begin{equation}\label{arcsine(1)}
\bm(\mathcal{E}_t)
=\int_0^{1/t} \frac{dx}{\pi \sqrt{x(1-x)}} \sim \frac2{\pi
\sqrt{t}}
\end{equation}
Observe on the other hand that
\begin{equation}\label{g_t > A (1)}
\bm(g_t\le A|\mathcal{E}_t)=\frac1{\bm(\mathcal{E}_t)}\bm(g_t \le A;
A_{g_t} \le 1; B_t >0).
\end{equation}

Also, yet another result of P. L\'evy tells us that if $(b_s,0\le
s\le 1)$ is a standard Brownian bridge, then
\begin{equation}\label{levy2}
\int_0^1 \indic{b_s \le 0} ds \overset{d}= \mathbf{U}
\end{equation}
a uniform random variable in $(0,1)$. (See, e.g, (3.9) in Chapter XII of \cite{revuz-yor}). Putting these pieces
together and using (\ref{g_t > A (1)}), it follows that for $A>1$,
\begin{eqnarray*}
\bm(g_t \le A; A^-_{g_t} \le 1; B_t >0) &=& \ds \frac12\int_0^{A/t}
\frac{dx}{\pi \sqrt{x(1-x)}} \cdot (1\wedge \frac1{xt}) \\
&=& \frac12\int_0^{1/t} \frac{dx}{\pi \sqrt{x(1-x)}} +
\frac12\int_{1/t}^{A/t} \frac{dx}{\pi \sqrt{x(1-x)}} \frac1{xt} \\
&\sim& \frac1{\pi \sqrt{t}} + \frac1{2t} \int_{1/t}^{A/t}
\frac{dx}{\pi x^{3/2}}\\
&\sim& \frac1{\pi\sqrt{t}}(2-1/\sqrt{A})
\end{eqnarray*}
(\ref{g_t > A}) now follows immediately from (\ref{g_t > A (1)})
and (\ref{arcsine(1)}). On the other hand, a similar computation yields for $0<A<1$,
\begin{equation}\label{gt>A2}
\bm(g_t<A | \cE_t) \sim \sqrt{A}/2.
\end{equation}
It is easy to see from (\ref{g_t > A})
that the conditional law of $B$ given $\mathcal{E}'_t$ converges
to the law of a transient process. This process can be described
via the following sample path decomposition. Let $g$ be a random
variable such that:
\begin{equation}\label{g}
P(g \in dy) =
\begin{cases}
\ds \frac14y^{-1/2} &\text{ if $0<y\le 1$}\\
\ds \frac14 y^{-3/2} dy & \text{ if $y>1$}.
\end{cases}
\end{equation}
This is obtained by differentiating respectively (\ref{gt>A2}) and (\ref{g_t
> A}). Let $(b_s,0 \le s\le 1)$ be an
independent Brownian bridge conditioned so that
\begin{equation}\label{cond br}
\int_0^g \indic{b_{t/g} \le 0} dt \le 1
\end{equation}
Then for $t\le g$ we put $X_t=g^{1/2}b_{t/g}$, and after time $g$,
we glue an independent 3-dimensional Bessel process. Note that with probability $1/2$, $g<1$, so in that case, the conditioning (\ref{cond br}) is trivial, and the total time accumulated by $(X_t,t \ge 0)$ in the negative half-line is uniform on $(0,g)$. By (\ref{levy2}), in that case we can thus write $A_\infty= \mathbf{U} g$,
 with $\mathbf{U}$ a uniform random variable  on (0,1) independent from $g$. On the other hand, if $g>1$, then $A_\infty$ is uniformly distributed on $(0,g)$, conditionally on being smaller than 1 (by (\ref{cond br})). Thus $A_\infty$ is uniformly distributed on (0,1) in that case.
 To finish the proof of Theorem \ref{A-}, it suffices to make a computation: for $0<u<1$,
 \begin{align*}
 \P(A_\infty <u) & =  \frac12\P(A_\infty <u |g>1) + \frac12\P(A_\infty <u |g<1) \\
 &= \frac12 u + \frac12 \left(\P(g<u|g<1) + \int_{u<y<1} \P(g\in dy |g<1) \frac{u}y\right)\\
 & = \frac12 u + \frac12(u^{1/2} + u(u^{-1/2}-1)) = u^{1/2} = \P(\mathbf{U}^2 <u).
 \end{align*}
 Hence the proof is finished.
\end{proof}

\section*{Acknowledgements}
We thank Ross Pinsky for some detailed comments improving on a first version of this paper, and Marc Yor for showing us a
draft of \cite{roynette-yor}. N.B. would like
to gratefully acknowledge PIMS and UBC, whose support was
determinant for the completion of this work, as well as the hospitality of the Weizmann Institute. We would like to thank a first anonymous referee, whose careful reading made us reformulate our results after finding some mistakes in an earlier version.

\end{document}